\documentclass[11pt,a4paper]{article}
\usepackage[T1]{fontenc}
\usepackage[utf8]{inputenc}
\usepackage{lmodern}
\usepackage[a4paper,margin=26mm,headheight=15pt]{geometry}
\usepackage{amsmath,amssymb,amsthm,mathtools,mathrsfs}
\usepackage{microtype,needspace,etoolbox}
\usepackage{booktabs,array,longtable}
\usepackage{hyperref}
\hypersetup{colorlinks=true,linkcolor=blue,citecolor=blue,urlcolor=blue,
 pdftitle={A Uniform Divisor-Comparison Method for Meromorphic Identities},
 pdfsubject={Fixed Weierstrass coordinates, uniform divisor-comparison proofs, and historical motivation},
 pdfauthor={Henning Wunderlich},bookmarksnumbered=true}
\numberwithin{equation}{section}
\newtheorem{theorem}{Theorem}[section]
\newtheorem{lemma}[theorem]{Lemma}
\newtheorem{proposition}[theorem]{Proposition}

\theoremstyle{definition}

\theoremstyle{remark}

\newtheorem*{normalization}{Normalization rule}
\newtheorem*{dctone}{DCT 1 --- General comparison}
\newtheorem*{dcttwo}{DCT 2 --- Finite-order comparison}
\newtheorem*{dctparity}{DCT 2(P) --- Common parity}
\newtheorem*{dctrecurrence}{DCT 2(R) --- Common recurrence}
\newtheorem*{dctthree}{DCT 3 --- Matching lattice transformations}
\newcommand{\C}{\mathbb C}
\newcommand{\Z}{\mathbb Z}

\newcommand{\OO}{\mathcal O(\C)}
\newcommand{\MM}{\mathcal M(\C)}
\newcommand{\DD}{\operatorname{Div}_{\mathrm{lf}}(\C)}
\newcommand{\HH}{\mathcal H}
\newcommand{\TT}{\mathscr T}
\newcommand{\KK}{\mathcal K}
\newcommand{\ii}{\mathrm i}
\newcommand{\dd}{\,\mathrm d}
\newcommand{\zero}{\mathbf 0}
\DeclareMathOperator{\Div}{div}
\DeclareMathOperator{\ord}{ord}

\newcommand{\lc}[2]{\operatorname{lc}_{#1}(#2)}

\newcommand{\proto}[1]{\begin{quote}\small #1\end{quote}}
\newcommand{\route}[1]{\par\smallskip\noindent\textit{DCT route.}\ #1\par\smallskip}
\newcommand{\step}[2]{\par\addvspace{5pt}\Needspace{3\baselineskip}
 \noindent\textbf{#1.\ #2.}\par\nobreak}
\newcommand{\represent}{\step{1}{Represent the two sides}}
\newcommand{\compare}{\step{2}{Compare the divisors}}
\newcommand{\applydct}{\step{3}{Apply DCT}}
\newcommand{\normalize}{\step{4}{Normalize the logarithmic coordinate}}
\newcommand{\finish}{\[ [h_s-h_t]=[0],\qquad s=t,\qquad f=q_s=q_t=g. \]}
\newenvironment{comparisonproof}{\par}{\par\hfill$\square$\par\medskip}
\pretocmd{\section}{\Needspace{8\baselineskip}}{}{}
\pretocmd{\subsection}{\Needspace{9\baselineskip}}{}{}
\BeforeBeginEnvironment{theorem}{\Needspace{6\baselineskip}}
\BeforeBeginEnvironment{lemma}{\Needspace{5\baselineskip}}
\BeforeBeginEnvironment{proposition}{\Needspace{5\baselineskip}}

\begin{document}
\title{A Uniform Divisor-Comparison Method for Meromorphic Identities}
\author{Henning Wunderlich\thanks{Offenbach am Main, Germany. Email: \texttt{henningwunderlich@t-online.de}}}
\date{}
\maketitle

\begin{abstract}
We present a uniform method for proving meromorphic function identities by
comparing complete zero and pole multisets. A fixed annular Weierstrass
normalization represents every nonzero meromorphic function uniquely as
\[
q_t=e^{h_t}\frac{p_{Z_t}}{p_{P_t}},\qquad
 t=([h_t],Z_t,P_t),\qquad [h_t]\in\mathcal O(\mathbb C)/(2\pi i\mathbb Z).
\]
Thirteen principal comparisons, together with their specializations and
consequences, use the same four steps: represent the two sides, compare
their divisors, apply a divisor-comparison theorem (DCT), and normalize
$h_s-h_t$. A small library packages classical finite-order, parity,
recurrence, and lattice-rigidity arguments for repeated use. The examples
span sine, Gamma, Barnes $G$, Bessel, completed zeta, theta, and elliptic
functions. Their breadth demonstrates reuse of one proof structure across
these families; the family-specific hypotheses and scalar calculations
remain explicit. The contribution is a methodological synthesis: shorter
proofs where available, and otherwise a modular organization of classical
arguments. The fixed factors supply common coordinates, rather than a new
uniqueness hypothesis. Historically, the viewpoint is motivated by the
nineteenth-century contrast between Riemann's global geometric approach to
complex function theory and Weierstrass's analytic construction of functions;
the present
method deliberately combines global divisor data with fixed Weierstrass
factors. An appendix develops the associated ring and field presentations,
separately from the ordinary operations used in the proofs.
\end{abstract}


\medskip
\section{Introduction and proof protocol}
\label{sec:purpose}
Many identities in complex analysis can be approached by the same question:
after their zeros and poles have been matched, what freedom remains?
The purpose of this article is to make that question the working method,
not merely a preliminary observation. We fix one product convention,
state a small library of divisor-comparison theorems (DCTs), and use the
same proof format throughout a substantial collection of identities.

Every nonzero meromorphic function is written in reduced coordinates
\[
 t=([h_t],Z_t,P_t),\qquad q_t=e^{h_t}\frac{p_{Z_t}}{p_{P_t}}.
\]
The product $p_D$ is prescribed for every locally finite multiset $D$ by
one annular rule. If $f=q_s$ and $g=q_t$ have the same zero and pole
multisets, the common factors cancel and
\begin{equation}\label{eq:central}
 \frac{f}{g}=e^{h_s-h_t}.
\end{equation}
The task is then to determine the single residual coordinate
$H=h_s-h_t$, never to calculate the two logarithmic coordinates separately.

\proto{\textbf{The proof template.}
\begin{align*}
\textbf{1. Represent:}\quad &f=q_s,\quad g=q_t.\\
\textbf{2. Compare:}\quad &Z_s=\cdots=Z_t,\quad P_s=\cdots=P_t.\\
\textbf{3. Apply DCT:}\quad &H:=h_s-h_t\text{ is entire, polynomial, or constant}.\\
\textbf{4. Normalize:}\quad &[H]=[0],\quad s=t,\quad f=g.
\end{align*}}

\paragraph{The contribution: a reusable proof organization.}
The central claim is methodological. Four comparisons use one parity DCT;
Gamma and Barnes multiplication use one recurrence DCT; theta and elliptic
identities use one lattice DCT. The reciprocal Gamma product uses the
underlying finite-order theorem and a two-coefficient normalization.
Each application visibly retains the four steps, including the final
coordinate equality. Their repetition is the point: after learning the
DCTs, the reader need only check the new divisor data, the stated analytic
hypotheses, and the normalization. The same format can be used in lectures
without presenting every function family as a new collection of tricks.

\paragraph{Why retain the full collection.}
The thirteen principal comparisons, including the separately displayed
Legendre specialization, and all their further consequences form the
main demonstration of reuse. They include simple and weighted zeros,
meromorphic pole sets, symmetric products, and additive and multiplicative
lattices. Agreement of the proof structure across these different
settings is the central demonstration. This is a
statement about the cases proved here, not a claim that any identity can be
settled from divisor data alone.

\paragraph{What the fixed factors add.}
The annular convention makes equal divisor coordinates produce literally
the same $p_Z/p_P$ in every application. There is no new choice of
convergence factors to reconcile each time the function family changes.
It is an explicit universal realization of the coordinates, not an
additional source of uniqueness: another fixed valid normalization would
support the same DCTs. The particular ceiling formula supplies a
convergence guarantee for every locally finite multiset; its value here
is a definite common convention. The DCT hypotheses and the last
normalization determine the residual unit.

\paragraph{Two meanings of simplification.}
A proof may be shorter once the general theorems are available.
Alternatively, it may become easier to organize, teach, or check because
the same argument is reused. We use \emph{simplification} in both senses,
without claiming to improve on the shortest known proof of every formula.
The independent zero, growth, and transformation lemmas are stated in
Appendix~\ref{app:inputs}; the nontrivial scalar computations are in
Appendix~\ref{app:scalars}. Moving these facts to reusable lemmas saves
repetition, not their mathematical content.

\paragraph{Complete divisors, not discrete sampling.}
Equality of divisors includes multiplicities and the absence of additional
zeros or poles. It is not agreement of values on a discrete set.
For unrestricted meromorphic functions, multiplication by $e^H$ preserves
the divisor for every entire $H$. The conclusion $f=g$ therefore uses all
three ingredients: complete divisor data, the DCT hypotheses, and the
prescribed normalization. Within such a normalized analytic class, the
divisor does determine the function.

\subsection{Historical motivation: Riemann's geometric approach and
Weierstrass's analytic construction}
\label{sec:riemann-history}
The method used here is motivated by a broad nineteenth-century change in
how complex functions could be understood.  Bottazzini and Gray devote
Chapter~5 of \emph{Hidden Harmony---Geometric Fantasies} to
``Riemann's Geometric Function Theory.''  Their account emphasizes
Riemann's global geometric approach to complex function theory, in which
Riemann surfaces, conformal mapping, Abelian integrals, branch points,
coverings, and related global analytic data became central organizing
structures, rather than merely consequences of a previously fixed
power-series representation \cite[Ch.~5]{BottazziniGray}.
This should not be read as the historical claim that Riemann formulated the
specific divisor-comparison theorems of this article.  Zeros and poles are
only one part of the broader Riemannian picture, which also includes branch
points, coverings, and monodromy.

The contrast with Weierstrass is especially useful for interpreting the
present construction.  Bottazzini documents how, partly in response to
Riemann's geometric methods, Weierstrass developed a systematic theory of
analytic functions on arithmetical and power-series foundations, and how
this methodological opposition remained influential well beyond their own
work \cite{BottazziniResponse}.  Ferreir\'os places Riemann's mathematics in
a still broader conceptual setting, emphasizing the interaction of
mathematics, geometry, physics, and philosophical questions in his style of
research \cite{FerreirosRiemann}.

The present paper deliberately combines the two traditions rather than
choosing between them.  Its \emph{data-first} step is Riemann-like in spirit:
for an identity $f=g$, one first compares global discrete data---the complete
zero and pole multisets---before manipulating explicit formulas.  Its
implementation is Weierstrassian: the prescribed factors $p_D$ give a
rigorous analytic representative for every locally finite divisor.  Once
these common factors have been fixed, a DCT identifies the residual
holomorphic freedom, and normalization removes it.  Thus the historical
motivation is a structural analogy and a synthesis of viewpoints, not an
attribution of the modern divisor formalism or the present four-step
protocol to Riemann himself.

\subsection{Classical context and related approaches}
\label{sec:context}
The underlying results are classical. Hadamard factorization and its
same-zero uniqueness consequence appear, for example, in McMullen
\cite[Theorem~3.14 and Corollary~3.15]{McMullen}. His proof of the
Weierstrass cubic already compares the complete zeros and poles and then
the leading Laurent coefficient \cite[Theorem~5.4]{McMullen}. Robbin's
notes place divisors alongside principal parts on Riemann surfaces
\cite[\S11]{Robbin}. These are direct antecedents of the comparison
arguments used here, not merely background on the functions involved.

A closely related modern treatment is P\'erez-Marco's characterization of
Euler's Gamma function within finite-order meromorphic functions
\cite{PerezMarcoGamma}, extended to higher Gamma functions in
\cite{PerezMarcoHigher}. In particular, the uniqueness argument in
\cite[\S1, Theorem~1.3 and its proof]{PerezMarcoHigher} reduces the ratio
of two candidates to an exponential polynomial, uses their recurrence,
and removes the remaining ambiguity by a real normalization. DCT~2(R)
packages this familiar mechanism for repeated comparison of already
defined functions. The present collection applies a single coordinate
convention and a fixed four-step format beyond the Gamma hierarchy, also
to parity-based and lattice-based identities. The formulas and standard
conventions are referenced individually, principally to the DLMF.

The article is consequently an expository and methodological synthesis,
not a claim of novelty for divisor uniqueness or for the classical
identities. The historical discussion likewise supplies motivation rather
than a claim that Riemann stated the modern divisor quotient or the DCTs in
this form. No bibliographical priority for the exact annular rule is
asserted. ``Canonical'' always means relative to the chosen coordinate,
origin, shells, and genus rule.

\paragraph{Organization and the algebraic companion.}
Sections~\ref{sec:representation}--\ref{sec:dct} give the coordinates and
DCTs; Section~\ref{sec:map} records their reuse across the full collection.
Sections~\ref{sec:parity-apps}--\ref{sec:lattice-apps} contain the thirteen
comparisons, and Section~\ref{sec:consequences} retains the derived
identities. Appendices~\ref{app:inputs}--\ref{app:conversion} supply
analytic input, scalar evaluations, and product conversions.
Appendix~\ref{app:algebra} develops the complete ring and field
presentations associated with the same coordinates. It is an algebraic
companion, not a prerequisite: all identity proofs use ordinary
meromorphic addition and multiplication under $\Phi(t)=q_t$.

\section{Fixed annular Weierstrass coordinates}
\label{sec:representation}
A discrete multiset $D$ is a locally finite sum
$D=\sum_a m_D(a)[a]$, with $m_D(a)\in\Z_{\geq0}$; multiplicities are finite.
Every compact set contains only finitely many points counted with
multiplicity.  Write $D_k$ for the part with $k\leq |a|<k+1$ and
$N_k(D)=\#D_k$.  Multiset union $\sqcup$ adds multiplicities.

For $p\geq0$, set
\[
 H_p(w)=\sum_{j=1}^{p}\frac{w^j}{j},\qquad H_0=0,
 \qquad E_p(w)=(1-w)e^{H_p(w)}.
\]
For $k\geq2$ and $N_k(D)>0$, fix exactly the rule
\begin{equation}\label{eq:genus}
 \boxed{p_k(D)=\left\lceil
 \frac{\log N_k(D)+2\log k}{\log k-\log\log(k+1)}
 \right\rceil-1.}
\end{equation}
All real logarithms here are natural.  Empty shells supply no factors.
Define
\begin{equation}\label{eq:annular}
 \boxed{p_D(z)=z^{m_D(0)}
 \prod_{\substack{a\in D\\0<|a|<2}}\left(1-\frac za\right)
 \prod_{k=2}^{\infty}\prod_{a\in D_k}
 E_{p_k(D)}(z/a),\qquad p_\varnothing=1.}
\end{equation}
The origin is extracted separately, and the finitely many remaining
small-shell factors have genus zero.

\begin{proposition}[The fixed product and its coordinates]\label{prop:coordinates}
The product \eqref{eq:annular} is entire, has exactly the zero multiset
$D$, and satisfies $z^{-m_D(0)}p_D(z)\to1$ at zero.  Every nonzero
meromorphic function on $\C$ has a unique reduced representation
\begin{equation}\label{eq:qt}
 q_t(z)=e^{h_t(z)}\frac{p_{Z_t}(z)}{p_{P_t}(z)},\qquad
 t=([h_t],Z_t,P_t),\quad Z_t\cap P_t=\varnothing,
\end{equation}
where $[h_t]\in\HH:=\OO/(2\pi\ii\Z)$, and $Z_t,P_t$ are its actual zeros
and poles.  Adjoin a separate symbol $\zero$ for the zero function.
\end{proposition}
\begin{proof}
Put $r_k=p_k(D)+1$.  The genus rule gives
\[
 N_k(D)\left(\frac{\log(k+1)}k\right)^{r_k}\leq k^{-2}.
\]
For $|w|\leq1/2$, the logarithm vanishing at zero satisfies
\[
 \log E_p(w)=-\sum_{j=p+1}^\infty\frac{w^j}{j},\qquad
 |\log E_p(w)|\leq2|w|^{p+1}.
\]
Fix $R$.  When $k\geq2R$ and $\log(k+1)\geq R$, the logarithmic
contribution of the $k$th shell on $|z|\leq R$ is at most $2/k^2$.
The tail therefore converges normally to a nonvanishing holomorphic
factor.  The finite initial part supplies exactly the prescribed zeros,
and all nonzero-point factors equal $1$ at zero.

For a given meromorphic $f$, local cancellation shows that
$f\,p_{P_f}/p_{Z_f}$ is an entire unit.  Every such unit $u$ has an entire
logarithm: take a primitive of $u'/u$ and choose its constant to obtain
$u=e^h$.  Two logarithms differ by a constant in $2\pi\ii\Z$.
The actual divisor first determines $Z_t,P_t$, and then this logarithm
uniquely determines $[h_t]$.
\end{proof}

\paragraph{A product on either side is treated as a function first.}
If a familiar identity uses genus-one or genus-two products, those are
perfectly legitimate definitions of $f$ or $g$.  Each entire or
meromorphic side is then represented by \emph{the same} rule
\eqref{eq:annular}.  In particular,
\[
 Z_s=Z_t=Z,\quad P_s=P_t=P
 \quad\Longrightarrow\quad
 f=e^{h_s}\frac{p_Z}{p_P},\quad
 g=e^{h_t}\frac{p_Z}{p_P}.
\]
This does not require calculating either logarithmic coordinate.
Appendix~\ref{app:conversion} gives the explicit conversion from a
fixed-genus product when one wants to display it.

\paragraph{No hidden change of genus.}
Even if $f$ has order one, its particular annular $h_t$ need not be linear.
The low-degree assertions below concern only $h_s-h_t$ after the common
annular factors cancel.  They do not silently replace the fixed $p_D$.

\section{Divisor-comparison theorems}
\label{sec:dct}
For the rest of the paper, $f=q_s$ and $g=q_t$ are nonzero functions in
reduced coordinates.  Write $H=h_s-h_t$.  All assertions about $H$ are
unchanged when its representative is shifted by $2\pi\ii n$.

\subsection{DCT 1: cancel the common canonical factors}
\label{sec:dct1}
\begin{dctone}
\[
 \boxed{Z_s=Z_t,\quad P_s=P_t
 \quad\Longleftrightarrow\quad f=e^H g\text{ for an entire }H.}
\]
When the divisor coordinates agree, $H=h_s-h_t$.
\end{dctone}
\begin{proof}
In the forward direction the two copies of $p_Z/p_P$ cancel in
\eqref{eq:qt}.  Conversely, multiplication by $e^H$ changes neither zeros
nor poles.  The implication also holds by direct local cancellation,
independently of a product choice.
\end{proof}

\subsection{DCT 2: finite order, with parity and recurrence rules}
\label{sec:dct2}
For a nonzero entire function $F$, its order is
\[
 \rho(F)=\limsup_{r\to\infty}
 \frac{\log\log\max\{e,M_F(r)\}}{\log r},\qquad
 M_F(r)=\max_{|z|=r}|F(z)|.
\]
For $0\leq\rho<\infty$, let $\mathcal E_\rho$ be the entire functions
of order at most $\rho$ (including $0$), and put
\[
 \KK_\rho=\{A/B:A,B\in\mathcal E_\rho,\ B\not\equiv0\}.
\]
This definition permits meromorphic functions without invoking a new
notion of meromorphic order.  Constants and polynomials have order zero.
Products and sums of entire functions do not increase the maximum of
their orders; hence $\KK_\rho$ is a field.  In particular
$\Gamma\in\KK_1$ because $1/\Gamma$ is entire of order one.

\begin{dcttwo}
If $f,g\in\KK_\rho^\times$ have the same zero and pole multisets, then
\begin{equation}\label{eq:dct2}
 \boxed{H=h_s-h_t\text{ is a polynomial of degree at most }
 \lfloor\rho\rfloor.}
\end{equation}
\end{dcttwo}
\begin{proof}
For entire functions with the same zeros, Hadamard factorization uses the
same minimal-genus product for both functions.  Its polynomial
exponents have degree at most $\lfloor\rho\rfloor$, and subtraction
gives the assertion; see \cite[Theorem 3.14]{McMullen}.
For $f=A/B$ and $g=C/D$, the entire functions $AD$ and $CB$ have order
at most $\rho$ and the same zero multiset, since
$\Div(AD)-\Div(CB)=\Div(f)-\Div(g)=0$.  Apply the entire result to
their quotient $f/g$.  Its entire logarithm differs from $H$ only by
$2\pi\ii n$.
\end{proof}

\begin{dctparity}
Assume DCT 2 with $\rho<2$.  If $f$ and $g$ have the same parity about
a point $c$, that is,
\[
 f(c+w)=\varepsilon f(c-w),\quad
 g(c+w)=\varepsilon g(c-w),\qquad \varepsilon\in\{1,-1\},
\]
then $H$ is constant.  In particular, in
$H(z)=a(z-c)+b$ one has
\begin{equation}\label{eq:parity}
 \boxed{a=0.}
\end{equation}
\end{dctparity}
\begin{proof}
DCT 2 gives degree at most one.  The quotient is even about $c$, so
$e^{aw+b}=e^{-aw+b}$.  Differentiate at $w=0$ to obtain $2ae^b=0$.
More generally, at arbitrary finite order, common parity makes $H$ an
even polynomial about $c$; it need not make it constant.
\end{proof}

\begin{dctrecurrence}
Assume DCT 2 with some finite $\rho$.  Suppose a real $T>0$ and a
nonzero meromorphic multiplier $A$ satisfy
\[
 f(z+T)=A(z)f(z),\qquad g(z+T)=A(z)g(z).
\]
If $f(x)/g(x)>0$ on a nonempty open real interval where the values are
finite and nonzero, then $H$ is constant.
\end{dctrecurrence}
\begin{proof}
Write $H=Q$, a polynomial.  The common recurrence gives
$e^{Q(z+T)-Q(z)}=1$, so $Q(z+T)-Q(z)=2\pi\ii k$ for one integer $k$.
A polynomial with constant first difference is linear: $Q=az+b$ and
$aT=2\pi\ii k$.  Positivity makes $\operatorname{Im}(ax+b)$ constant
modulo $2\pi$ on an interval, so $\operatorname{Im}a=0$.  But $a$ is
purely imaginary by the preceding relation.  Thus $a=0$.
\end{proof}

The polynomial-periodicity argument above is also the uniqueness
mechanism in \cite[\S1]{PerezMarcoHigher}; its role here is a reusable
comparison rule. The positivity hypothesis removes a genuine ambiguity.
A common recurrence alone allows the nonconstant periodic unit
$e^{2\pi\ii kz/T}$.  Likewise, parity alone at order two allows
$e^{az^2}$.

\subsection{DCT 3: the same lattice multipliers}
\label{sec:dct3}
\begin{dctthree}
Suppose $f,g$ have the same zeros and poles.  Let $\lambda_1,\lambda_2$
be linearly independent over $\mathbb R$, and suppose
\[
 f(z+\lambda_j)=J_j(z)f(z),\qquad
 g(z+\lambda_j)=J_j(z)g(z),\qquad j=1,2,
\]
with exactly the same holomorphic nowhere-zero multipliers.  Then
\begin{equation}\label{eq:dct3}
 \boxed{H=h_s-h_t\text{ is constant}.}
\end{equation}
\end{dctthree}
\begin{proof}
DCT 1 makes $f/g=e^H$ an entire unit.  Its two multipliers are $1$,
so it is bounded on a fundamental parallelogram and hence on $\C$.
Liouville's theorem makes it a nonzero constant.  An entire logarithm
of a constant is constant as well.
\end{proof}

Elliptic functions are the case $J_1=J_2=1$.  Equivalently, on any compact
connected Riemann surface, meromorphic functions with the same divisor
have constant quotient.  On $\C^\times$ with a multiplicative period
$0<|q|<1$, lift by $z=e^w$; the periods become $2\pi\ii$ and $\log q$.
The plane representation applies to this lifted divisor, not to a
multiset accumulating at the omitted origin.

\subsection{Normalize the last coordinate}
\label{sec:normalization}
For a nonzero meromorphic $f$, define its leading local coefficient at
$a$ by
\[
 \lc{a}{f}=\lim_{z\to a}(z-a)^{-\ord_a f}f(z)\ne0.
\]
This includes values at ordinary points, derivatives at simple zeros,
and leading Laurent coefficients at poles.

\begin{normalization}
If a DCT gives $H=b$ constant, then at any point $a$,
\begin{equation}\label{eq:normalization}
 e^b=\frac{\lc{a}{f}}{\lc{a}{g}}.
\end{equation}
In particular, one matching nonzero local coefficient proves
$[h_s]=[h_t]$ and $f=g$.  More generally, any one common nonzero linear
coefficient functional gives the same conclusion; this includes a
Laurent or Fourier coefficient.

If DCT 2 gives $\deg H\leq N$, remove the common local power
$(z-a)^m$ and match the Taylor coefficients of the two resulting units
through degree $N$.  These $N+1$ coefficients imply $f=g$.
\end{normalization}
\begin{proof}
The constant case follows from $f=e^bg$ and linearity.  In the polynomial
case write $F=(z-a)^{-m}f$, $G=(z-a)^{-m}g$.  The coefficient conditions
mean $F/G=1+O((z-a)^{N+1})$.  The local logarithm vanishing at $a$ is
$O((z-a)^{N+1})$ and differs from the global polynomial $H$ by a
constant in $2\pi\ii\Z$.  A polynomial of degree at most $N$ with
this vanishing is zero.
\end{proof}

\proto{\textbf{Normalized uniqueness.}
In a class with fixed DCT hypotheses and a fixed admissible
normalization, the complete divisor determines the function:
\[
 (Z_f,P_f)=(Z_g,P_g)\quad\Longleftrightarrow\quad f=g.
\]
For example, normalized even entire functions of order below two are
uniquely determined by their zero multisets.  The analytic class and
normalization are part of the statement, not optional extra conclusions.}

\section{Uniform application protocol}
\label{sec:map}
All constants in a claimed identity are included in $g$ before the
comparison starts. Thus the final task is always $[h_s-h_t]=[0]$.
The complete collection below has thirteen principal comparisons, each with
a full four-step proof. Table~\ref{tab:reuse} shows which general theorem
is reused and which normalization finishes each comparison.

\begin{table}[htbp]
\centering\small
\caption{The same four steps in every principal comparison.}
\label{tab:reuse}
\begin{tabular}{@{}>{\raggedright\arraybackslash}p{.43\linewidth}>{\raggedright\arraybackslash}p{.17\linewidth}>{\raggedright\arraybackslash}p{.32\linewidth}@{}}
\toprule
\textbf{Identity} & \textbf{DCT} & \textbf{Normalization}\\
\midrule
Sine product (\S\ref{ex:sine}) & 2(P) & Leading coefficient at $0$\\[3pt]
Gamma reflection (\S\ref{ex:reflection}) & 2(P) & Residue at $0$\\[3pt]
Bessel products, including $J_0$ (\S\ref{ex:bessel}) & 2(P) & Value at $0$\\[3pt]
Centered $\xi$-product (\S\ref{ex:xi}) & 2(P) & Value at the center\\[3pt]
Reciprocal Gamma product (\S\ref{ex:gamma-product}) & 2 & Two local coefficients\\[3pt]
Gauss multiplication (\S\ref{ex:gauss}) & 2(R) & Residue at $0$\\[3pt]
Legendre duplication (\S\ref{ex:legendre}) & 2(R) & Residue at $0$\\[3pt]
Barnes duplication (\S\ref{ex:barnes}) & 2(R) & Value at $1/2$\\[3pt]
Jacobi triple product (\S\ref{ex:jacobi}) & 3 & One Laurent coefficient\\[3pt]
First theta product (\S\ref{ex:theta}) & 3 & One Fourier coefficient\\[3pt]
Weierstrass cubic (\S\ref{ex:cubic}) & 3 & Leading Laurent coefficient\\[3pt]
Sigma addition (\S\ref{ex:sigma-add}) & 3 & Leading Laurent coefficient\\[3pt]
Sigma duplication (\S\ref{ex:sigma-dup}) & 3 & Leading Taylor coefficient\\
\bottomrule
\end{tabular}
\end{table}

\paragraph{The same obligations in every proof.}
Step~1 must establish genuine nonzero meromorphic functions in the plane
variable, including convergence or a lift from $\C^\times$ when needed.
Step~2 must identify the \emph{entire} zero and pole multisets after
cancellation. Step~3 checks the hypotheses of the cited DCT, rather than
assuming that matching divisors already gives equality. Step~4 exhibits
the normalization required to make the logarithmic class zero.
Family-specific considerations fill these four positions; they do not
replace the proof scheme with an unrelated argument.

\paragraph{Dependencies and the cost of the input.}
Each application cites the relevant lemma in Appendix~\ref{app:inputs}.
The Bessel zero theorem precedes its product; the theta zero count
precedes either series--product identity; and the elliptic half-period
zero count precedes the cubic. The functional equation of $\xi$ is
assumed independently of its zero product. Gamma reflection supplies the
finite scalar used for Gauss multiplication. Legendre duplication is both
its specialization and a separate recurrence-DCT comparison; the Barnes
application then uses Legendre's formula. The Barnes
half-value and the theta coefficient calculation are explicit scalar
lemmas in Appendix~\ref{app:scalars}. These dependencies make the
applications modular without concealing the analytic work.

\clearpage
\section{Common parity: four applications of the same DCT}
\label{sec:parity-apps}
\subsection{Euler's sine product}
\label{ex:sine}
The classical identity is \cite[\S4.22]{DLMFTrig}
\begin{equation}\label{eq:sine}
 \sin\pi z=\pi z\prod_{n=1}^{\infty}\left(1-\frac{z^2}{n^2}\right).
\end{equation}
\route{DCT 2(P): same zeros, order below two, common oddness; one coefficient.}
\begin{comparisonproof}
\represent
Set
\[
 f(z)=\sin\pi z=q_s(z),\qquad
 g(z)=\pi z\prod_{n\geq1}(1-z^2/n^2)=q_t(z),
\]
with $s=([h_s],Z_s,\varnothing)$ and
$t=([h_t],Z_t,\varnothing)$.  Product convergence and order at most one
follow from Lemma~\ref{lem:paired-growth}.
\compare
All zeros are simple, and
\[
 Z_s=\sum_{n\in\Z}[n]=[0]+\sum_{n\geq1}([n]+[-n])=Z_t,
 \qquad P_s=P_t=\varnothing.
\]
Thus both triples use the same $p_{\Z}$ from \eqref{eq:annular}.
\applydct
Both functions are odd and belong to $\KK_1$.  DCT 2(P) gives
$h_s-h_t=b$.
\normalize
At their common simple zero,
\[
 \lc{0}{f}=\pi=\lc{0}{g}.
\]
The normalization rule gives $e^b=1$, hence
\finish
\end{comparisonproof}

\clearpage
\subsection{Euler's Gamma reflection formula}
\label{ex:reflection}
We compare the meromorphic sides directly, retaining their pole data:
\begin{equation}\label{eq:reflection}
 \Gamma(z)\Gamma(1-z)=\frac{\pi}{\sin\pi z}.
\end{equation}
This is the usual reflection identity \cite[\S5.5(ii)]{DLMFGammaRelations}.
\route{DCT 2(P): same poles, finite-order quotients, common symmetry; one residue.}
\begin{comparisonproof}
\represent
Set
\[
 f(z)=\Gamma(z)\Gamma(1-z)=q_s(z),\qquad
 g(z)=\pi/\sin\pi z=q_t(z),
\]
where $s=([h_s],\varnothing,P_s)$ and
$t=([h_t],\varnothing,P_t)$.  The Gamma data are in
Lemma~\ref{lem:gamma-data}.
\compare
There are no zeros.  The pole multisets are
\[
 P_s=\{0,-1,-2,\ldots\}\sqcup\{1,2,3,\ldots\}
     =\Z=P_t,
 \qquad Z_s=Z_t=\varnothing,
\]
all poles simple.  Both canonical denominators are $p_{\Z}$.
\applydct
Both sides lie in $\KK_1$ and are even about $1/2$:
$f(1-z)=f(z)$ and $g(1-z)=g(z)$.  DCT 2(P) gives $h_s-h_t=b$.
\normalize
The recurrence $z\Gamma(z)=\Gamma(1+z)$ and $\Gamma(1)=1$ give
\[
 \lc{0}{f}=\lim_{z\to0}z\Gamma(z)\Gamma(1-z)=1
 =\lim_{z\to0}\frac{\pi z}{\sin\pi z}=\lc{0}{g}.
\]
Consequently $e^b=1$, and
\finish
\end{comparisonproof}
In particular, positivity gives $\Gamma(1/2)=\sqrt\pi$.  No integral is
evaluated inside the four-step proof; the independent analytic facts
about Gamma remain prerequisites.

\clearpage
\subsection{Bessel products, including \texorpdfstring{$J_0$}{J0}}
\label{ex:bessel}
For real $\nu\geq0$, the entire normalized function is
\[
 B_\nu(z)=\sum_{k=0}^{\infty}
 \frac{(-1)^k\Gamma(\nu+1)}{k!\Gamma(k+\nu+1)}(z/2)^{2k}.
\]
Its positive zeros are denoted $j_{\nu,k}$.  The identity is
\begin{equation}\label{eq:bessel}
 B_\nu(z)=\prod_{k=1}^{\infty}(1-z^2/j_{\nu,k}^2).
\end{equation}
It is the normalized form of the classical Bessel product
\cite[\S10.21(iii)]{DLMFBesselZeros}.
\route{DCT 2(P): same paired zeros, order below two, common evenness; one value.}
\begin{comparisonproof}
\represent
Set $f=B_\nu=q_s$ and $g=\prod_{k\geq1}(1-z^2/j_{\nu,k}^2)=q_t$,
with $s=([h_s],Z_s,\varnothing)$ and
$t=([h_t],Z_t,\varnothing)$.
Lemma~\ref{lem:bessel-data} proves the zero classification without this
product, and Lemma~\ref{lem:paired-growth} supplies its convergence and
order bound.
\compare
The zero theorem and the individual product factors give
\[
 Z_s=\sum_{k\geq1}([j_{\nu,k}]+[-j_{\nu,k}])=Z_t,
 \qquad P_s=P_t=\varnothing.
\]
All multiplicities are one; both sides therefore use the same annular
product for this multiset.
\applydct
Both functions are even and belong to $\KK_1$.  DCT 2(P) gives
$h_s-h_t=b$.
\normalize
\[
 \lc{0}{f}=B_\nu(0)=1=g(0)=\lc{0}{g}.
\]
Thus $e^b=1$, and
\finish
\end{comparisonproof}
On a consistent branch this is
\[
 J_\nu(z)=\frac{(z/2)^\nu}{\Gamma(\nu+1)}
             \prod_{k\geq1}(1-z^2/j_{\nu,k}^2).
\]
The case $\nu=0$ requires no branch or prefactor:
\begin{equation}\label{eq:J0}
 J_0(z)=\prod_{k\geq1}(1-z^2/j_{0,k}^2).
\end{equation}

\clearpage
\subsection{The centered product for the completed zeta function}
\label{ex:xi}
Let
\[
 \xi(s)=\tfrac12s(s-1)\pi^{-s/2}\Gamma(s/2)\zeta(s),\qquad
 X(z)=\xi(\tfrac12+z).
\]
The functional equation, entireness, order bound, and nonzero central
value are independent input in Lemma~\ref{lem:xi-data}; see
\cite{DLMFZetaDef,DLMFZetaFE,DLMFZetaZeros}.
Let $\alpha=\rho-1/2$ run over the centered nontrivial zeros with
multiplicity.  The claimed identity is
\begin{equation}\label{eq:xi-centered}
 X(z)=X(0)\prod_{\{\alpha,-\alpha\}}
                      (1-z^2/\alpha^2)^{m(\alpha)},
\end{equation}
where one representative is chosen from each nonzero pair.
\route{DCT 2(P): same centered zeros, order below two, common evenness; one value.}
\begin{comparisonproof}
\represent
Set $f=X=q_s$ and
$g=X(0)\prod_{\{\alpha,-\alpha\}}(1-z^2/\alpha^2)^{m(\alpha)}=q_t$,
with $s=([h_s],Z_s,\varnothing)$ and
$t=([h_t],Z_t,\varnothing)$.  Lemma~\ref{lem:paired-growth} gives
normal convergence and order at most one for $g$.
\compare
The functional equation pairs $\alpha$ with $-\alpha$, and $X(0)\ne0$.
Thus, counting all multiplicities,
\[
 Z_s=\sum_{\{\alpha,-\alpha\}}m(\alpha)([\alpha]+[-\alpha])=Z_t,
 \qquad P_s=P_t=\varnothing.
\]
These are divisors in the centered variable $z$, so the annular products
are formed in that same variable.
\applydct
Both functions are even and belong to $\KK_1$.  DCT 2(P) gives
$h_s-h_t=b$.
\normalize
\[
 \lc{0}{f}=\xi(1/2)=\lc{0}{g}\ne0.
\]
Hence $e^b=1$, and
\finish
\end{comparisonproof}
In the original variable the same identity is
\begin{equation}\label{eq:xi-product}
 \xi(s)=\xi(1/2)\prod_{\{\rho,1-\rho\}}
       \left(1-\frac{(s-1/2)^2}{(\rho-1/2)^2}\right),
\end{equation}
with multiplicities included.  Substitution of the definition of $\xi$
gives the corresponding formula for the completed expression in $\zeta$.
This proof neither assumes the Riemann hypothesis nor proves the
functional equation from zero symmetry derived from that same equation.

\clearpage
\section{Finite order and recurrence: normalize the same coordinate}
\label{sec:recurrence-apps}
\subsection{The reciprocal Gamma product}
\label{ex:gamma-product}
The standard formula \cite[\S5.8]{DLMFGammaProduct} is
\begin{equation}\label{eq:gamma-product}
 \Gamma(z)^{-1}=ze^{\gamma z}
          \prod_{n\geq1}(1+z/n)e^{-z/n}.
\end{equation}
\route{DCT 2: same zeros and order one; two local coefficients.}
\begin{comparisonproof}
\represent
Set
\[
 f(z)=\Gamma(z)^{-1}=q_s(z),\qquad
 g(z)=ze^{\gamma z}\prod_{n\geq1}E_1(-z/n)=q_t(z),
\]
with $s=([h_s],Z_s,\varnothing)$ and
$t=([h_t],Z_t,\varnothing)$.  The product estimate in
Lemma~\ref{lem:gamma-data} gives order at most one for $g$.
\compare
The two simple-zero multisets agree:
\[
 Z_s=\{0,-1,-2,\ldots\}=Z_t,\qquad P_s=P_t=\varnothing.
\]
Both sides are therefore represented with the same prescribed $p_{Z_s}$,
although the displayed classical product uses genus one.
\applydct
Both sides belong to $\KK_1$.  DCT 2 gives $h_s-h_t=az+b$.
\normalize
The independent local Gamma data and the product expansion give
\[
 f(z)=z+\gamma z^2+O(z^3),\qquad
 g(z)=z+\gamma z^2+O(z^3).
\]
After removing the common factor $z$, the two coefficients through degree
one agree.  The polynomial normalization rule yields
\finish
\end{comparisonproof}
When Gamma is instead \emph{defined} by this product, this application
is a coordinate verification, not an independent derivation of the
definition.  Reflection and multiplication are unaffected by that choice.

\clearpage
\subsection{Gauss multiplication}
\label{ex:gauss}
For a positive integer $n$, the identity is
\begin{equation}\label{eq:gauss}
 \prod_{r=0}^{n-1}\Gamma(z+r/n)
  =(2\pi)^{(n-1)/2}n^{1/2-nz}\Gamma(nz).
\end{equation}
Positive-base powers use real logarithms.  This is Gauss's classical
formula \cite[\S5.5(iii)]{DLMFGammaRelations}.
\route{DCT 2(R): same poles, common recurrence and positivity; one residue.}
\begin{comparisonproof}
\represent
Write $c_n(z)=(2\pi)^{(n-1)/2}n^{1/2-nz}$ and set
\[
 f(z)=\prod_{r=0}^{n-1}\Gamma(z+r/n)=q_s(z),\qquad
 g(z)=c_n(z)\Gamma(nz)=q_t(z),
\]
with $s=([h_s],\varnothing,P_s)$ and
$t=([h_t],\varnothing,P_t)$.
\compare
There are no zeros, and unique division $j=n\ell+r$ gives
\[
 P_s=\bigsqcup_{r=0}^{n-1}\{-\ell-r/n:\ell\geq0\}
     =\{-j/n:j\geq0\}=P_t,
 \qquad Z_s=Z_t=\varnothing.
\]
All poles are simple.  The two canonical denominators are literally the
same $p_{P_s}$, using $N_k(P_s)=n$ for $k\geq2$.
\applydct
Both functions are in $\KK_1$ and are positive for real $z>0$.
The Gamma recurrence and $c_n(z+1/n)=c_n(z)/n$ give
\[
 f(z+1/n)=z f(z),\qquad g(z+1/n)=z g(z).
\]
DCT 2(R), with $T=1/n$, therefore gives $h_s-h_t=b$.
\normalize
The finite scalar identity in Lemma~\ref{lem:gamma-scalar}, obtained
from reflection and roots of unity, gives
\[
 \lc{0}{f}=\prod_{r=1}^{n-1}\Gamma(r/n)
  =(2\pi)^{(n-1)/2}n^{-1/2}
  =\frac{c_n(0)}n=\lc{0}{g}.
\]
Thus $e^b=1$, and
\finish
\end{comparisonproof}
There is no separate periodic-limit or Stirling argument here.  The
same recurrence DCT will be used for Barnes duplication.

\clearpage
\subsection{Legendre duplication}
\label{ex:legendre}
For clarity, its short proof is displayed in the same form, even though
it is also the specialization $n=2$ of \eqref{eq:gauss}:
\begin{equation}\label{eq:legendre}
 \Gamma(z)\Gamma(z+1/2)=2^{1-2z}\sqrt\pi\,\Gamma(2z).
\end{equation}
\route{DCT 2(R): the same recurrence argument; the residue is $\sqrt\pi$.}
\begin{comparisonproof}
\represent
Set $f=\Gamma(z)\Gamma(z+1/2)=q_s$ and
$g=2^{1-2z}\sqrt\pi\,\Gamma(2z)=q_t$, with
$s=([h_s],\varnothing,P_s)$ and $t=([h_t],\varnothing,P_t)$.
\compare
\[
 P_s=\{-m:m\geq0\}\sqcup\{-m-1/2:m\geq0\}
     =\{-j/2:j\geq0\}=P_t,
 \qquad Z_s=Z_t=\varnothing.
\]
\applydct
Both sides are in $\KK_1$, positive on the positive real axis, and
satisfy $F(z+1/2)=zF(z)$.  DCT 2(R) gives $h_s-h_t=b$.
\normalize
Both residues at zero are $\sqrt\pi$, using
$\Gamma(1/2)=\sqrt\pi$ from reflection.  Therefore $e^b=1$ and
\finish
\end{comparisonproof}

\clearpage
\subsection{Barnes \texorpdfstring{$G$}{G} duplication}
\label{ex:barnes}
Use the value-normalized identity
\begin{equation}\label{eq:barnes}
 \begin{split}
 G(z)G(z+1/2)^2G(z+1)
  ={}&2^{-2z^2+3z-1}\pi^zG(1/2)^2G(2z).
 \end{split}
\end{equation}
The Barnes input is its standard product and recurrence, not a
multiplication identity \cite[\S5.17]{DLMFBarnes}.
\route{DCT 2(R): same weighted zeros, common recurrence and positivity; one value.}
\begin{comparisonproof}
\represent
Set $f(z)=G(z)G(z+1/2)^2G(z+1)=q_s(z)$ and
\[
 g(z)=2^{-2z^2+3z-1}\pi^zG(1/2)^2G(2z)=q_t(z),
\]
with $s=([h_s],Z_s,\varnothing)$ and
$t=([h_t],Z_t,\varnothing)$.
\compare
At $-m$, the multiplicity of $f$ is $(m+1)+m=2m+1$; at
$-m-1/2$ it is $2(m+1)=2m+2$.  Therefore
\[
 Z_s=\sum_{j\geq0}(j+1)[-j/2]=Z_t,\qquad P_s=P_t=\varnothing.
\]
The common annular product has shell multiplicity $4k+3$ for $k\geq2$.
\applydct
Lemma~\ref{lem:barnes-data} puts both functions in $\KK_2$ and gives
positivity for $z>0$.  The Barnes recurrence and the already proved
Legendre formula yield, for both $F=f$ and $F=g$,
\[
 F(z+1)=A(z)F(z),\qquad
 A(z)=2^{1-4z}\pi\,\Gamma(2z)\Gamma(2z+1).
\]
DCT 2(R) gives $h_s-h_t=b$.
\normalize
At $z=1/2$, the recurrence and $G(1)=1$ give
\[
 f(1/2)=\sqrt\pi\,G(1/2)^2=g(1/2)\ne0.
\]
Thus $e^b=1$, and
\finish
\end{comparisonproof}
The optional scalar evaluation in Lemma~\ref{lem:barnes-scalar} turns
\eqref{eq:barnes} into the conventional form
\begin{equation}\label{eq:barnes-standard}
 \begin{split}
 G(z)G(z+1/2)^2G(z+1)
  ={}&e^{1/4}A^{-3}2^{-2z^2+3z-11/12}
        \pi^{z-1/2}G(2z),
 \end{split}
\end{equation}
where $A$ is Glaisher's constant.  Its evaluation is a scalar task,
separate from the divisor proof of the dependence on $z$.

\clearpage
\section{Matching lattices: five applications of the same DCT}
\label{sec:lattice-apps}
\subsection{Jacobi's triple product}
\label{ex:jacobi}
For $0<|q|<1$, use
$(a;q)_\infty=\prod_{r\geq0}(1-aq^r)$ and put
\[
 F(z;q)=\sum_{n\in\Z}(-1)^nq^{n(n-1)/2}z^n,\qquad
 P(z;q)=(z;q)_\infty(q/z;q)_\infty.
\]
The identity \cite[\S17.8]{DLMFqTriple} is
\begin{equation}\label{eq:jacobi}
 F(z;q)=(q;q)_\infty P(z;q),\qquad z\in\C^\times.
\end{equation}
\route{DCT 3: lift to the plane, compare complete lattice zeros and multipliers;
match one Laurent coefficient.}
\begin{comparisonproof}
\represent
Choose $\ell=\log q$ and lift by $z=e^w$.  Set
\[
 f(w)=F(e^w;q)=q_s(w),\qquad
 g(w)=(q;q)_\infty P(e^w;q)=q_t(w),
\]
with $s=([h_s],Z_s,\varnothing)$ and
$t=([h_t],Z_t,\varnothing)$.
Both are nonzero entire functions in $w$.
\compare
The independent zero-count lemma, Lemma~\ref{lem:theta-data}, and the
product factors give simple zeros exactly at
\[
 Z_s=\ell\Z+2\pi\ii\Z=Z_t,\qquad P_s=P_t=\varnothing.
\]
This is locally finite in the $w$-plane, where the prescribed annular
factors are formed.
\applydct
Both sides obey
\[
 f(w+2\pi\ii)=f(w),\qquad f(w+\ell)=-e^{-w}f(w),
\]
with the same equations for $g$.  Since $\Re\ell<0$, the periods are
real-linearly independent.  DCT 3 gives $h_s-h_t=b$.
\normalize
The zeroth Laurent coefficient in $z=e^w$ is $1$ on each side:
\[
 [z^0]F(z;q)=1,
 \qquad [z^0]\{(q;q)_\infty P(z;q)\}=1.
\]
The second equality is the independently proved scalar coefficient
lemma, Lemma~\ref{lem:q-scalar}.  Linearity gives $e^b=1$, hence
\finish
\end{comparisonproof}
Surjectivity of the exponential map then gives \eqref{eq:jacobi} on
$\C^\times$.  The full zero classification precedes DCT; it is not
inferred from the product identity being proved.

\clearpage
\subsection{Jacobi's first theta product}
\label{ex:theta}
Let $\Im\tau>0$, $q=e^{\pi\ii\tau}$, and
$q^{1/4}=e^{\pi\ii\tau/4}$.  Define
\[
 \theta_1(z\mid\tau)=2\sum_{n\geq0}(-1)^nq^{(n+1/2)^2}
                                     \sin((2n+1)z).
\]
The product formula is \cite[\S20.5]{DLMFThetaProducts}
\begin{equation}\label{eq:theta}
 \theta_1(z\mid\tau)=2q^{1/4}\sin z
 \prod_{n\geq1}(1-q^{2n})(1-2q^{2n}\cos2z+q^{4n}).
\end{equation}
\route{DCT 3: same simple lattice zeros and exact multipliers; one Fourier coefficient.}
\begin{comparisonproof}
\represent
Set $f=\theta_1(\,\cdot\mid\tau)=q_s$ and let $g=q_t$ be the full
right-hand side of \eqref{eq:theta}, including its scalar factor.
Here $s=([h_s],Z_s,\varnothing)$ and
$t=([h_t],Z_t,\varnothing)$.
\compare
Lemma~\ref{lem:theta-data} and product splitting give
\[
 Z_s=\pi\Z+\pi\tau\Z=Z_t,\qquad P_s=P_t=\varnothing,
\]
with all zeros simple.  Both sides therefore use the same lattice
product \eqref{eq:annular}.
\applydct
For both $F=f$ and $F=g$,
\[
 F(z+\pi)=-F(z),\qquad
 F(z+\pi\tau)=-q^{-1}e^{-2\ii z}F(z).
\]
DCT 3 gives $h_s-h_t=b$.
\normalize
The coefficient of $e^{-\ii z}$ is $\ii q^{1/4}$ on both sides.
For $f$ this follows from its defining series; for $g$ it follows from
Lemma~\ref{lem:q-scalar} with base $q^2$, as detailed there.
This common coefficient is nonzero, so $e^b=1$ and
\finish
\end{comparisonproof}
This application uses the same scalar lemma as the triple product; it
does not need to reprove compactness or to leave a nome-dependent
constant undetermined.

\clearpage
\subsection{The Weierstrass cubic equation}
\label{ex:cubic}
Fix $\Lambda=2\omega_1\Z+2\omega_2\Z$ with
$\Im(\omega_2/\omega_1)>0$, and put $\omega_3=\omega_1+\omega_2$,
$e_j=\wp(\omega_j)$.  The formula is
\begin{equation}\label{eq:cubic}
 (\wp'(z))^2=4\prod_{j=1}^3(\wp(z)-e_j).
\end{equation}
The definitions and the target equation are recorded in
\cite{DLMFEllipticDef,DLMFEllipticODE}. For the classical divisor proof,
see \cite[Theorem~5.4]{McMullen}. The zero information used below
is proved before the cubic in Lemma~\ref{lem:elliptic-data}.
\route{DCT 3: same double half-period zeros and sixth-order poles; one Laurent coefficient.}
\begin{comparisonproof}
\represent
Set $f=(\wp')^2=q_s$ and $g=4\prod_j(\wp-e_j)=q_t$, with
$s=([h_s],Z_s,P_s)$ and $t=([h_t],Z_t,P_t)$.
\compare
The three half-period classes are simple zeros of $\wp'$; each
$\wp-e_j$ has a double zero at the corresponding class.  Hence
\[
 Z_s=2\bigsqcup_{j=1}^3(\omega_j+\Lambda)=Z_t,
 \qquad P_s=6\Lambda=P_t.
\]
There are no other zeros or poles, by the independent torus zero count.
\applydct
Both sides are $\Lambda$-periodic.  DCT 3 gives $h_s-h_t=b$.
\normalize
Using $\wp(z)\sim z^{-2}$ and $\wp'(z)\sim-2z^{-3}$,
\[
 \lc{0}{f}=4=\lc{0}{g}.
\]
Thus $e^b=1$, and
\finish
\end{comparisonproof}
The expanded form is
\begin{equation}\label{eq:cubic-expanded}
 (\wp')^2=4\wp^3-g_2\wp-g_3,\quad
 g_2=60\sum_{\lambda\ne0}\lambda^{-4},\quad
 g_3=140\sum_{\lambda\ne0}\lambda^{-6}.
\end{equation}
Its coefficients follow by expanding \eqref{eq:cubic} at zero, as in
Lemma~\ref{lem:elliptic-data}.

\clearpage
\subsection{Sigma addition}
\label{ex:sigma-add}
For $v\notin\Lambda$, the identity is
\begin{equation}\label{eq:sigma-add}
 \frac{\sigma(z+v)\sigma(z-v)}{\sigma(z)^2\sigma(v)^2}
                 =\wp(v)-\wp(z).
\end{equation}
This is the classical sigma addition formula
\cite[\S23.10]{DLMFEllipticAddition}.
\route{DCT 3: same translated zero cosets and double lattice poles; one Laurent coefficient.}
\begin{comparisonproof}
\represent
Fix $v$ and set
\[
 f(z)=\frac{\sigma(z+v)\sigma(z-v)}{\sigma(z)^2\sigma(v)^2}=q_s(z),
 \qquad g(z)=\wp(v)-\wp(z)=q_t(z),
\]
with $s=([h_s],Z_s,P_s)$ and $t=([h_t],Z_t,P_t)$.
\compare
Lemma~\ref{lem:elliptic-data} gives
\[
 Z_s=(v+\Lambda)\sqcup(-v+\Lambda)=Z_t,
 \qquad P_s=2\Lambda=P_t.
\]
If $2v\in\Lambda$, the zero cosets coincide and count twice on both
sides.  Since $v\notin\Lambda$, there is no zero-pole cancellation.
\applydct
The sigma multipliers cancel between numerator and denominator; both
sides are $\Lambda$-periodic.  DCT 3 gives $h_s-h_t=b$.
\normalize
Oddness of $\sigma$, $\sigma(z)\sim z$, and $\wp(z)\sim z^{-2}$ give
\[
 \lc{0}{f}=-1=\lc{0}{g}.
\]
Therefore $e^b=1$, and
\finish
\end{comparisonproof}

\clearpage
\subsection{Sigma duplication}
\label{ex:sigma-dup}
The identity is
\begin{equation}\label{eq:sigma-dup}
 \sigma(2z)=-\wp'(z)\sigma(z)^4.
\end{equation}
It is recorded in \cite[\S23.10]{DLMFEllipticAddition}.  Although it is
also a limit of sigma addition, its direct DCT proof fits the same template.
\route{DCT 3: same simple half-lattice zeros and matching sigma multipliers; one derivative.}
\begin{comparisonproof}
\represent
Set $f(z)=\sigma(2z)=q_s(z)$ and $g(z)=-\wp'(z)\sigma(z)^4=q_t(z)$,
with $s=([h_s],Z_s,\varnothing)$ and
$t=([h_t],Z_t,\varnothing)$; the apparent poles of $g$ are removable.
\compare
At a lattice point, $g$ has order $4-3=1$; at a nonzero half-period
class it has order one from $\wp'$.  Hence
\[
 Z_s=\tfrac12\Lambda
 =\Lambda\sqcup\bigsqcup_{j=1}^3(\omega_j+\Lambda)=Z_t,
 \qquad P_s=P_t=\varnothing.
\]
Every zero is simple.
\applydct
The standard sigma transformations yield, for both $F=f$ and $F=g$,
\[
 F(z+2\omega_j)=e^{8\eta_j(z+\omega_j)}F(z),\qquad j=1,2.
\]
DCT 3 gives $h_s-h_t=b$.
\normalize
At their common simple zero,
\[
 \lc{0}{f}=2=\lc{0}{g},
\]
since $-\wp'(z)\sigma(z)^4\sim2z$.  Thus $e^b=1$ and
\finish
\end{comparisonproof}

\clearpage
\section{Consequences of already normalized identities}
\label{sec:consequences}
Once a four-step proof gives $[h_s-h_t]=[0]$, ordinary differentiation,
algebra, and specialization preserve the identity.  The following
formulas are presented as \emph{consequences}, not as new proof methods
or as divisor comparisons with missing zero information.

\subsection{Cotangent and cosecant-squared partial fractions}
Logarithmically differentiating \eqref{eq:sine} gives
\begin{equation}\label{eq:cot}
 \pi\cot\pi z=\frac1z+
 \sum_{n\geq1}\left(\frac1{z-n}+\frac1{z+n}\right).
\end{equation}
Differentiating once more gives
\begin{equation}\label{eq:csc}
 \frac{\pi^2}{\sin^2\pi z}=\sum_{n\in\Z}\frac1{(z-n)^2}.
\end{equation}
These are the classical partial fractions \cite[\S4.22]{DLMFTrig}.
For $|z|\leq R$ and $n>2R$, a paired summand in \eqref{eq:cot} is
$2z/(z^2-n^2)=O_R(n^{-2})$.  The differentiated series has the same
summable bound on compact sets avoiding the poles.  Normal convergence
therefore justifies both differentiations.  No independent principal-part
proof has been inserted into the DCT sequence.

\subsection{The zeta logarithmic derivatives}
Logarithmic differentiation of \eqref{eq:xi-centered} gives
\begin{equation}\label{eq:xi-log}
 \frac{X'(z)}{X(z)}=
 \sum_{\{\alpha,-\alpha\}}m(\alpha)
 \left(\frac1{z-\alpha}+\frac1{z+\alpha}\right).
\end{equation}
The paired summand is $2z/(z^2-\alpha^2)$, bounded by
$C_R|\alpha|^{-2}$ for $|\alpha|>2R$.  The zero summability from
Lemma~\ref{lem:paired-growth} justifies differentiation.  Thus
\begin{equation}\label{eq:xi-log-s}
 \frac{\xi'(s)}{\xi(s)}=
 \sum_{\{\rho,1-\rho\}}
 \left(\frac1{s-\rho}+\frac1{s-(1-\rho)}\right),
\end{equation}
where pairs retain multiplicity.  Differentiating the definition of $\xi$
then yields
\begin{equation}\label{eq:zeta-log}
 \begin{split}
 \frac{\zeta'(s)}{\zeta(s)}={}&
 \sum_{\{\rho,1-\rho\}}
 \left(\frac1{s-\rho}+\frac1{s-(1-\rho)}\right)\\
 &-\frac1s-\frac1{s-1}+\frac12\log\pi-\frac12\psi(s/2),
 \qquad \psi=\Gamma'/\Gamma.
 \end{split}
\end{equation}
All equalities are meromorphic; paired summation is part of the formulas.

\subsection{Weierstrass addition and duplication}
Let $\mathcal Z=\sigma'/\sigma$, so $\mathcal Z'=-\wp$.
Logarithmically differentiate the already normalized sigma addition
identity \eqref{eq:sigma-add} first in $u$ and then in $v$, and add the
results.  One obtains
\begin{equation}\label{eq:Z-add}
 \mathcal Z(u+v)-\mathcal Z(u)-\mathcal Z(v)
 =\frac12\frac{\wp'(u)-\wp'(v)}{\wp(u)-\wp(v)}.
\end{equation}
This is an equality on the open set where the denominators are nonzero,
and hence a meromorphic equality.

For completeness, put $x=\wp(u)$, $y=\wp(v)$,
$a=\wp'(u)$, $b=\wp'(v)$, and $d=x-y$.  Differentiating
\eqref{eq:Z-add} in $u$ gives
\[
 \wp(u+v)=x-
 \frac{\wp''(u)d-(a-b)a}{2d^2}.
\]
The proved cubic gives
$a^2-b^2=4(x^3-y^3)-g_2(x-y)$ and
$\wp''(u)=6x^2-g_2/2$.  Substituting and simplifying gives the classical
addition theorem \cite[\S23.10]{DLMFEllipticAddition}:
\begin{equation}\label{eq:wp-add}
 \wp(u+v)=-\wp(u)-\wp(v)+\frac14
 \left(\frac{\wp'(u)-\wp'(v)}{\wp(u)-\wp(v)}\right)^2.
\end{equation}
Letting $v\to u$ at ordinary points gives
\begin{equation}\label{eq:wp-dup}
 \wp(2u)=-2\wp(u)+\frac14
                    \left(\frac{\wp''(u)}{\wp'(u)}\right)^2.
\end{equation}
The identities extend meromorphically across exceptional arguments.
Thus the cubic and sigma DCT proofs supply the addition and duplication
formulas without a separate principal-part comparison argument.

\section{What the collection establishes}
\label{sec:conclusion}
The thirteen principal comparisons instantiate one proof protocol. Four
invoke common parity, three common recurrence and positivity, and five
matching lattice transformations; the reciprocal Gamma product uses
finite order and two local coefficients. Each starts from $f=q_s$,
$g=q_t$, compares the complete divisors, and ends with $[h_s-h_t]=[0]$.
The specializations and differentiated identities in
Section~\ref{sec:consequences} extend the collection without a new DCT.

\paragraph{Uniformity is the demonstrated benefit.}
The input changes---multiplicities, growth, recurrences, lattice
multipliers, or scalar coefficients---but the division of the proof does
not. A function-family lemma is reused within its family, while a DCT is
reused across families. The breadth of the worked cases makes this reuse
inspectable. Some comparisons are short; others primarily gain a common
organization with explicit dependencies. Neither the Bessel zero theorem
nor the theta scalar calculation becomes cost-free, and no universal
reduction in proof length or determination from discrete samples is claimed.

\paragraph{The fixed convention and the residual comparison.}
If another fixed product has $\widetilde p_D=e^{u_D}p_D$, then
\[
 \widetilde h_t=h_t-u_{Z_t}+u_{P_t}.
\]
Once the divisors agree,
$[\widetilde h_s-\widetilde h_t]=[h_s-h_t]$. Thus the annular rule fixes
definite common factors without adding a uniqueness hypothesis. It also
makes clear why the individual logarithmic coordinates need not be
computed.

\paragraph{The algebraic companion.}
Appendix~\ref{app:algebra} develops the full ring and field presentation,
including the multiplication discrepancy, compatible corrections,
distributivity, the entire-function subring and fraction field, and
restricted-divisor closure. Its transported evaluation $\Psi_\beta$ is
separate from $\Phi(t)=q_t$, which is used in every identity proof.

The common conclusion is
\proto{\centering
\textbf{Complete divisor data + DCT hypotheses + normalization}\par
\smallskip
$\Longrightarrow\quad [h_s]=[h_t]\quad\Longrightarrow\quad f=g.$}
The classical principles supply correctness; the fixed coordinates and
the complete series of four-step comparisons supply the uniform
organization developed here.

\clearpage
\appendix
\addtocontents{toc}{\protect\setcounter{tocdepth}{1}}

\section{Independent analytic input}
\label{app:inputs}
The lemmas in this appendix supply facts about the functions or products
before any target identity uses them.  They make the dependencies visible;
they are not claimed to have no cost.  In particular, neither a Bessel
product nor the Weierstrass cubic is used to establish the zero data on
which its own DCT proof depends.

\subsection{Paired products: convergence and growth}
\begin{lemma}\label{lem:paired-growth}
Let a locally finite multiset be symmetric under $a\mapsto-a$, with
multiplicity $m$ at zero.  Suppose its nonzero counting function satisfies
$n(r)=O_\varepsilon(r^{\rho+\varepsilon})$ for every $\varepsilon>0$,
where $0\leq\rho<2$.  Then
\[
 P(z)=z^m\prod_{\{a,-a\}}(1-z^2/a^2)^{m(a)}
\]
converges locally uniformly, has exactly that multiset, and has order at
most $\rho$.  In particular this holds for a symmetric zero multiset of
an entire function of order at most $\rho<2$.
\end{lemma}
\begin{proof}
Choose $\rho<s<2$.  Counting with multiplicity and integrating by parts
gives $\sum_{a\ne0}|a|^{-2}<\infty$.  This proves normal convergence
and exact zeros of the paired product.  Up to its finite zero factor,
\[
 \log |P(z)|\leq\sum_{\{a,-a\}}m(a)\log(1+r^2/|a|^2)
 \leq C+2r^2\int_{a_0}^{\infty}\frac{n(t)}{t(t^2+r^2)}\dd t
 =O(r^s),
\]
where $a_0>0$ lies below every nonzero modulus; a finite multiset causes
no difficulty.  The last integral after scaling is bounded because
$0<s<2$.  Taking $s$ arbitrarily close to $\rho$ proves the order bound.
For zeros of an entire function, Jensen's formula gives the stated
counting estimate; see \cite[\S3]{McMullen}.  Thus this product bound does
not assume the product identity it will be used to prove.
\end{proof}

\subsection{Gamma: the data used by the comparisons}
\begin{lemma}\label{lem:gamma-data}
For the standard Gamma function, $1/\Gamma$ is entire of order one with
simple zeros $0,-1,-2,\ldots$, and no others.  Moreover,
\[
 \Gamma(z+1)=z\Gamma(z),\quad \Gamma(1)=1,\quad
 \Gamma'(1)=-\gamma,\quad \Gamma(x)>0\ (x>0).
\]
Consequently $\Gamma(z)^{-1}=z+\gamma z^2+O(z^3)$ at zero.  The
function $ze^{\gamma z}\prod_{n\geq1}E_1(-z/n)$ is entire of order at
most one with the same zeros.
\end{lemma}
\begin{proof}
The Gamma assertions are classical input from analytic continuation,
the elementary recurrence, and the standard growth theory, taken
independently of the target product in Section~\ref{ex:gamma-product};
see \cite{DLMFGammaDef,DLMFGammaValues,DLMFGammaAsymptotic,McMullen}.
The local expansion follows by writing $1/\Gamma(z)=z/\Gamma(1+z)$.
For the displayed product, normal convergence follows from
$\sum n^{-2}<\infty$.  For $|z|\leq r$, split at $n=2r$.
The initial logarithmic upper bound is
$\sum_{n\leq2r}(\log(1+r/n)+r/n)=O(r\log(r+2))$, and the tail is
$O(\sum_{n>2r}r^2/n^2)=O(r)$.  Hence the order is at most one.
\end{proof}
The first assertions may equivalently be adopted from the established
Gamma theory.  If that theory starts from the reciprocal product, the
reciprocal-product application is a normalization example rather than an
independent proof.  No circularity of this kind is used to prove reflection
or multiplication.

\subsection{Barnes: zeros, growth, and positivity}
\begin{lemma}\label{lem:barnes-data}
The standard Barnes function obeys
\[
 G(z+1)=\Gamma(z)G(z),\qquad G(1)=1,
\]
is entire of order at most two, has a zero of multiplicity $m+1$ at
$-m$ for $m\geq0$, and no other zeros, and is positive for positive real
arguments.
\end{lemma}
\begin{proof}
Use the standard product and recurrence \cite[(5.17.1), (5.17.3)]{DLMFBarnes}:
\[
 G(1+z)=(2\pi)^{z/2}e^{-z(z+1)/2-\gamma z^2/2}
                         \prod_{k\geq1}E_2(-z/k)^k.
\]
The product converges normally since $\sum k/k^3<\infty$ and gives
exactly the asserted zeros.  For real $z>-1$, every factor and the real
exponential prefactor are positive; normal convergence of the logarithmic
tail makes the nonzero limit positive.  This proves positivity for all
positive arguments of $G$.
For $|z|\leq r$, the tail $k>2r$ contributes at most
$C\sum_{k>2r}k(r/k)^3=O(r^2)$.  The initial factors contribute at most
\[
 \sum_{k\leq2r}\left(k\log(1+r/k)+r+\frac{r^2}{2k}\right)
 =O(r^2\log(r+2)).
\]
The prefactor has order at most two, establishing the required bound.
\end{proof}

\subsection{Bessel: a zero theorem before the product}
\begin{lemma}\label{lem:bessel-data}
For real $\nu\geq0$, $B_\nu$ is an even entire function of order at most
one, with $B_\nu(0)=1$.  All its zeros are real, nonzero, simple, and
form infinitely many pairs $\pm j_{\nu,k}$, where
$0<j_{\nu,1}<j_{\nu,2}<\cdots$.
\end{lemma}
\begin{proof}
Its defining power series is the normalized Bessel series
\cite{DLMFBesselDef}.  Since $(\nu+1)_k\geq k!$ and
$(2k)!\leq4^k(k!)^2$,
\[
 |B_\nu(z)|\leq\sum_{k\geq0}\frac{(|z|/2)^{2k}}{(k!)^2}
 \leq\cosh |z|\leq e^{|z|}.
\]
Termwise differentiation gives
$B_\nu''+(2\nu+1)z^{-1}B_\nu'+B_\nu=0$ away from zero.
For any nonzero zero $a$, put $y(x)=B_\nu(ax)$ on $[0,1]$.  Then
\[
 (x^{2\nu+1}y')'+a^2x^{2\nu+1}y=0,\quad y(0)=1,\quad y(1)=0.
\]
Multiplying by $\overline y$ and integrating by parts gives
\[
 a^2\int_0^1x^{2\nu+1}|y|^2\dd x
       =\int_0^1x^{2\nu+1}|y'|^2\dd x>0.
\]
The boundary terms vanish, since $y'(x)=O(x)$ at zero.  Thus $a^2>0$
and $a$ is real.  A multiple zero at $a\ne0$ would make both Cauchy
data of the differential equation zero, forcing $B_\nu\equiv0$.

If there were only finitely many zeros, Hadamard factorization would give
$B_\nu=e^{az+b}P$ with $P$ an even polynomial.  Common evenness forces
$a=0$, contradicting the infinitely many nonzero coefficients of the
defining series.  Thus the zeros are infinite and can be enumerated as
stated.  No Bessel product is used in this argument.
\end{proof}

\subsection{Completed zeta: the independent analytic input}
\begin{lemma}\label{lem:xi-data}
The completed function $\xi$ is entire of order at most one, satisfies
$\xi(s)=\xi(1-s)$, and has exactly the nontrivial zeta zeros with their
multiplicities.  Its central value $\xi(1/2)$ is nonzero.
\end{lemma}
\begin{proof}
Analytic continuation and the functional equation are classical independent
inputs \cite{DLMFZetaDef,DLMFZetaFE,DLMFZetaZeros}.
They cancel the trivial zeros and apparent poles in the completed
expression and identify its remaining zeros.  The order upper bound can
be obtained without a zero product: the Dirichlet series on $\Re s\geq2$,
Euler--Maclaurin continuation in a fixed vertical strip, Stirling's
estimate for the Gamma factor, and the functional equation on the other
half-plane give
\[
 \log\max\{1,M_\xi(r)\}=O(r\log(r+2)).
\]
This is the usual order-one estimate; Euler--Maclaurin continuation and
the asymptotic formula are in \cite{DLMFZetaDef,DLMFGammaAsymptotic}.
Finally, for real $0<\sigma<1$,
\[
 \zeta(\sigma)=\frac{\eta(\sigma)}{1-2^{1-\sigma}},\qquad
 \eta(\sigma)=\sum_{n\geq1}(-1)^{n-1}n^{-\sigma}>0.
\]
The last strict inequality follows by pairing successive positive and
negative terms.  Hence $\zeta(1/2)\ne0$ and so $\xi(1/2)\ne0$.
\end{proof}

\subsection{Theta zero data from transformations, not the triple product}
\begin{lemma}[One zero per cell]\label{lem:one-cell}
Let $T,L\in\C$ satisfy $\Im(L/T)>0$.  Suppose a nonzero entire function
$U$ obeys
\[
 U(w+T)=aU(w),\qquad
 U(w+L)=b e^{-2\pi\ii w/T}U(w),\qquad a,b\ne0.
\]
Then $U$ has exactly one zero, counted with multiplicity, in each
fundamental cell modulo $T\Z+L\Z$.
\end{lemma}
\begin{proof}
Choose the boundary of a translated cell to avoid zeros, and write
$V=U'/U$.  Logarithmic differentiation gives
\[
 V(w+T)=V(w),\qquad V(w+L)=V(w)-2\pi\ii/T.
\]
The two edges parallel to $L$ cancel in the contour integral of $V$.
The difference on the two edges parallel to $T$ integrates to $2\pi\ii$.
The argument principle therefore counts one zero and no poles.
\end{proof}

\begin{lemma}\label{lem:theta-data}
The functions $F(z;q)$ and $P(z;q)$ in Section~\ref{ex:jacobi} converge
normally on compact subsets of $\C^\times$, satisfy
\[
 U(qz;q)=-z^{-1}U(z;q),
\]
and have exactly the simple zeros $q^m$, $m\in\Z$.  The defining series
of $\theta_1$ is entire, has exactly the simple zeros
$\pi\Z+\pi\tau\Z$, and satisfies
\[
 \theta_1(z+\pi)=-\theta_1(z),\qquad
 \theta_1(z+\pi\tau)=-q^{-1}e^{-2\ii z}\theta_1(z).
\]
The product side of \eqref{eq:theta} has the same transformations and
simple zero multiset.
\end{lemma}
\begin{proof}
Quadratic decay of $q^{n(n-1)/2}$ gives convergence of both tails of $F$;
geometric decay gives normal convergence of $P$.  Reindexing $F$ and
cancelling first factors in $P$ give their common transformation.
The zeroth Laurent coefficient of $F$ is $1$, so $F$ is not identically
zero.  Pairing $n$ with $1-n$ proves $F(1;q)=0$.  Its lift $F(e^w;q)$
satisfies Lemma~\ref{lem:one-cell} with $T=2\pi\ii$, $L=\ell=\log q$,
$a=1$, $b=-1$; indeed $\Im(\ell/(2\pi\ii))>0$.
The zero at zero and its translates exhaust its zeros by the cell count,
and they are simple.  The product $P$ has exactly these simple zeros by
its individual factors.

For theta, the series converges normally, is odd and nonzero, and its
bilateral reindexing gives the displayed transformations.  Apply
Lemma~\ref{lem:one-cell} with $T=\pi$, $L=\pi\tau$, $a=-1$,
$b=-q^{-1}$.  Oddness supplies the zero at zero, so its lattice
translates are all the zeros and are simple.  These properties agree
with the standard theta conventions \cite{DLMFThetaDef}.

To check the product side independently, set $Q=q^2$, $w=e^{2\ii z}$,
and
\[
 U(z)=\sin z\prod_{n\geq1}
             (1-q^{2n}e^{2\ii z})(1-q^{2n}e^{-2\ii z}).
\]
Product splitting gives
\begin{equation}\label{eq:U-to-P}
 U(z)=\frac{\ii}{2}e^{-\ii z}P(w;Q).
\end{equation}
The transformation of $P$ gives exactly the two theta transformations
for $U$.  Its factors also give precisely the stated lattice of simple
zeros.  Multiplication by the scalar in \eqref{eq:theta} changes none
of this data.  No series-product identity has been used to count the
series' zeros.
\end{proof}

\subsection{Lattice functions before their addition and cubic identities}
\begin{lemma}\label{lem:elliptic-data}
For $\Lambda=2\omega_1\Z+2\omega_2\Z$, define
\begin{align}
 \wp(z)&=z^{-2}+\sum_{\lambda\in\Lambda\setminus\{0\}}
             \bigl((z-\lambda)^{-2}-\lambda^{-2}\bigr),\label{eq:wp-definition}\\
 \sigma(z)&=z\prod_{\lambda\in\Lambda\setminus\{0\}}E_2(z/\lambda).
 \label{eq:sigma-definition}
\end{align}
Then $\sigma$ is odd, has simple zeros exactly at $\Lambda$, and
$\sigma(z)=z+O(z^3)$.  The function $\wp$ is even and elliptic, with
one double pole per cell; $\wp'$ is odd and elliptic with simple zeros
exactly at the three nonzero half-period classes.  The $e_j$ are distinct,
$\wp-e_j$ has a double zero only at $\omega_j$ modulo $\Lambda$, and
for $v\notin\Lambda$, $\wp(z)-\wp(v)$ has zeros $v,-v$ modulo $\Lambda$, with the coinciding
case counted twice.  Finally
\[
 \sigma(z+2\omega_j)=-e^{2\eta_j(z+\omega_j)}\sigma(z),\qquad
 (\sigma'/\sigma)'=-\wp,
\]
for constants $\eta_j$.
\end{lemma}
\begin{proof}
Lattice-point counting gives $\#\{\lambda:|\lambda|\leq r\}=O(r^2)$
and $\sum_{\lambda\ne0}|\lambda|^{-3}<\infty$.  The two definitions
therefore converge normally away from the indicated poles, and the
product has exactly its stated zeros.  Pairing opposite points gives
the two parities and, near zero,
\begin{equation}\label{eq:wp-local}
 \wp(z)=z^{-2}+a z^2+b z^4+O(z^6),\quad
 a=3\sum_{\lambda\ne0}\lambda^{-4},\quad
 b=5\sum_{\lambda\ne0}\lambda^{-6}.
\end{equation}
These definitions and local normalizations are the standard ones
\cite{DLMFEllipticDef,DLMFEllipticLaurent}.

Differentiation gives the absolutely convergent periodic sum
$\wp'(z)=-2\sum_{\lambda\in\Lambda}(z-\lambda)^{-3}$.
Thus $\wp(z+2\omega_j)-\wp(z)$ is constant.  Evaluate at
$z=-\omega_j$ using evenness to make it zero.  An elliptic function
has equal total zero and pole multiplicities on its torus: the opposite
boundary contributions of its logarithmic derivative cancel in the
argument principle.  Since $\wp'$ has one triple pole and vanishes at
each nonzero half-period by oddness, these are exactly three simple zeros.
It follows that $\wp-e_j$ has a double zero at $\omega_j$.  Its one double
pole leaves no further zeros, and forces the three $e_j$ to be distinct.
Evenness and the same zero count give the assertion for general $v$.

Logarithmic differentiation of \eqref{eq:sigma-definition} gives
$\mathcal Z'= -\wp$ for $\mathcal Z=\sigma'/\sigma$.
Consequently $\mathcal Z(z+2\omega_j)-\mathcal Z(z)=2\eta_j$ is constant.
Local cancellation makes $\sigma(z+2\omega_j)/\sigma(z)$ an entire
unit, with logarithmic derivative $2\eta_j$; it is $C_je^{2\eta_jz}$.
Evaluate at $z=-\omega_j$ and use oddness to find
$C_j=-e^{2\eta_j\omega_j}$.  This proves the transformations without an
addition theorem.
\end{proof}

After the cubic has been proved by DCT, its expanded coefficients follow
from \eqref{eq:wp-local}.  In a monic cubic in $\wp$, comparison of the
$z^{-4}$ coefficient gives $e_1+e_2+e_3=0$, and comparison of the next
two coefficients gives $g_2=20a$ and $g_3=28b$, exactly the sums in
\eqref{eq:cubic-expanded}.  This subsequent coefficient calculation is
not used in the preceding zero count.

\section{Scalar normalizations}
\label{app:scalars}
The main comparisons leave a constant, or in the Gamma product a linear
polynomial.  This appendix records the less immediate scalar data rather
than hiding them in the fourth step.

\subsection{The finite Gamma normalization for Gauss multiplication}
\begin{lemma}\label{lem:gamma-scalar}
For every positive integer $n$,
\begin{equation}\label{eq:gamma-scalar}
 \prod_{r=1}^{n-1}\Gamma(r/n)=(2\pi)^{(n-1)/2}n^{-1/2}.
\end{equation}
This follows from reflection, not from Gauss multiplication.
\end{lemma}
\begin{proof}
The case $n=1$ is the empty product.  For $n\geq2$, factor
$(x^n-1)/(x-1)$ over the nontrivial $n$th roots of unity and evaluate
at $x=1$.  Taking absolute values gives
\[
 \prod_{r=1}^{n-1}2\sin(\pi r/n)=n.
\]
Let $B=\prod_{r=1}^{n-1}\Gamma(r/n)>0$.  Reversing the factors and using
the proved reflection identity \eqref{eq:reflection} gives
\[
 B^2=\prod_{r=1}^{n-1}\Gamma(r/n)\Gamma(1-r/n)
 =\prod_{r=1}^{n-1}\frac{\pi}{\sin(\pi r/n)}
 =\frac{(2\pi)^{n-1}}n.
\]
The positive square root proves the formula.  Thus the dependency is
reflection $\to$ scalar normalization $\to$ Gauss multiplication,
not a circle.
\end{proof}

\subsection{One common coefficient lemma for the two theta products}
\begin{lemma}\label{lem:q-scalar}
For $0<|q|<1$,
\begin{equation}\label{eq:P-constant}
 [z^0]\{(z;q)_\infty(q/z;q)_\infty\}=(q;q)_\infty^{-1}.
\end{equation}
It follows that the product side of \eqref{eq:jacobi} has constant Laurent
coefficient $1$, and that the product side of \eqref{eq:theta} has
coefficient $\ii q^{1/4}$ at $e^{-\ii z}$.
\end{lemma}
\begin{proof}
First derive Euler's expansion without using a triple product:
\begin{equation}\label{eq:euler-q}
 (t;q)_\infty=\sum_{r\geq0}
     \frac{(-1)^r q^{r(r-1)/2}}{(q;q)_r}t^r,
 \qquad (q;q)_r=\prod_{j=1}^r(1-q^j).
\end{equation}
Indeed, if $A(t)=(t;q)_\infty=\sum a_rt^r$, then
$A(t)=(1-t)A(qt)$ and $A(0)=1$ imply
$(1-q^r)a_r=-q^{r-1}a_{r-1}$.  Iteration gives \eqref{eq:euler-q},
whose quadratic exponents ensure entire convergence.  This is the
standard Euler expansion \cite{DLMFqBinomial}.

Multiply the two expansions on any circle in $\C^\times$, where both
are absolutely convergent.  The constant coefficient is
\[
 [z^0]P(z;q)=\sum_{r\geq0}\frac{q^{r^2}}{(q;q)_r^2}.
\]
To evaluate this scalar, partition an integer partition diagram by its
largest upper-left square of size $r\times r$.  Removing that square
leaves independently a diagram with at most $r$ rows on its right and
one with row lengths at most $r$ below.  Each remainder has generating
series $(q;q)_r^{-1}$, after transposing the first diagram.
The square contributes $q^{r^2}$, so
\begin{equation}\label{eq:durfee}
 \sum_{r\geq0}\frac{q^{r^2}}{(q;q)_r^2}
     =\prod_{j\geq1}(1-q^j)^{-1}.
\end{equation}
This is an identity of formal power series.  Absolute convergence for
$|q|<1$ makes it analytic: the left side is bounded by a constant times
$\sum |q|^{r^2}$ and the partition series by its convergent product at
$|q|$.  Restricted-partition generating functions are described in
\cite{DLMFPartitions}.  Equations \eqref{eq:P-constant} and
\eqref{eq:durfee} follow.

For theta put $Q=q^2$, $w=e^{2\ii z}$.  By \eqref{eq:U-to-P}, its full
product side is
\[
 g(z)=\ii q^{1/4}e^{-\ii z}(Q;Q)_\infty P(w;Q).
\]
The coefficient of $e^{-\ii z}$ is therefore $\ii q^{1/4}$.
The defining sine series has precisely the same coefficient.  This
calculation normalizes both identities without assuming either.
\end{proof}

\subsection{The optional Barnes half-value and conventional constant}
\begin{lemma}\label{lem:barnes-scalar}
With $A$ the Glaisher--Kinkelin constant,
\begin{equation}\label{eq:Ghalf}
 G(1/2)=A^{-3/2}\pi^{-1/4}e^{1/8}2^{1/24}.
\end{equation}
Consequently the value-normalized identity \eqref{eq:barnes} has exactly
the conventional form \eqref{eq:barnes-standard}.
\end{lemma}
\begin{proof}
For this optional scalar evaluation use the independent standard asymptotic
\cite[(5.17.5), (5.17.7)]{DLMFBarnes}:
\begin{equation}\label{eq:barnes-asym}
 \log G(t+1)=\left(\frac{t^2}{2}-\frac1{12}\right)\log t
 -\frac34t^2+\frac t2\log(2\pi)+K+o(1),\qquad
 K=\frac1{12}-\log A,
\end{equation}
as $t\to+\infty$.  Let $\mathcal A(t)$ denote the displayed main
expression and $L_0=\log(2\pi)$.  For fixed $a>0,b\in\mathbb R$,
Taylor expansion gives
\begin{align*}
 \mathcal A(ax+b)={}&\frac{a^2x^2}{2}(\log x+\log a-3/2)\\
 &+abx(\log x+\log a-1)+\frac{aL_0x}{2}\\
 &+\left(\frac{b^2}{2}-\frac1{12}\right)(\log x+\log a)
 +\frac{bL_0}{2}+K+o(1).
\end{align*}
Apply this at $(a,b)=(1,-1),(1,-1/2),(1,0),(2,-1)$, with weights
$1,2,1,-1$.  Thus for
$R(x)=G(x)G(x+1/2)^2G(x+1)/G(2x)$,
\[
 \log R(x)=-2x^2\log2+x(3\log2+\log\pi)
       +3K-\tfrac12\log\pi-\tfrac{11}{12}\log2+o(1).
\]
But the already proved value-normalized formula says exactly
\[
 R(x)=C\,2^{-2x^2+3x}\pi^x,\qquad C=G(1/2)^2/2.
\]
Comparison of constants gives
$C=e^{3K}\pi^{-1/2}2^{-11/12}$.
Take the positive square root of $2C$, since $G(1/2)>0$, to obtain
\eqref{eq:Ghalf}.  Substituting gives \eqref{eq:barnes-standard}.
\end{proof}
This asymptotic input evaluates a scalar.  It is not needed for the
four-step proof of the value-normalized Barnes identity.

\section{Reconciliation with the prescribed annular factors}
\label{app:conversion}
The applications never change the fixed $p_D$.  Here are explicit
formulas connecting their familiar products to that normalization.
For a signed divisor $D=\sum_a d(a)[a]$, write
\[
 \kappa_D=\frac{p_{D^+}}{p_{D^-}},\qquad
 D^+(a)=\max(d(a),0),\quad D^-(a)=\max(-d(a),0).
\]
Let $g_D(a)$ be the genus actually used at $a\ne0$: it is computed from
$D^+$ for a positive coefficient, from $D^-$ for a negative one, and is
zero for $0<|a|<2$.

\begin{proposition}[Fixed-genus conversion]\label{prop:conversion}
Fix an integer $m\geq0$ and suppose
\[
 \sum_{a\ne0}|d(a)|\,|a|^{-(m+1)}<\infty.
\]
Then the following series converges locally uniformly to an entire
function:
\begin{equation}\label{eq:beta-m}
 \beta_m(D)(z)=\sum_{a\ne0}d(a)
                  \bigl(H_m(z/a)-H_{g_D(a)}(z/a)\bigr).
\end{equation}
It satisfies
\begin{equation}\label{eq:fixed-genus-conversion}
 s_m(D)(z):=z^{d(0)}\prod_{a\ne0}E_m(z/a)^{d(a)}
                      =e^{\beta_m(D)(z)}\kappa_D(z).
\end{equation}
\end{proposition}
\begin{proof}
On a compact disk and for sufficiently large $|a|$,
$H_m(z/a)-H_{g_D(a)}(z/a)=\log E_m(z/a)-\log E_{g_D(a)}(z/a)$.
The first tail is absolutely summable by the hypothesis and the second
by the annular estimate.  Their difference converges normally; the
finitely many remaining terms are polynomials.  Exponentiate and use
the primary-factor formula to prove \eqref{eq:fixed-genus-conversion}.
Positive and negative products are treated separately before taking
the meromorphic quotient.
\end{proof}

For example, the products for sine, Bessel, and centered $\xi$ are
$s_1(D)$ after pairing $a,-a$.  Thus a candidate $g=C s_1(D)$ is
\emph{already} in the exact coordinates
\[
 t=([\log C+\beta_1(D)],D,\varnothing).
\]
The DCT proof establishes the same logarithmic class for the other side;
it does not need to calculate $\beta_1(D)$ explicitly.
Likewise $\sigma=e^{\beta_2(\Lambda)}p_\Lambda$, since
$\sum_{\lambda\ne0}|\lambda|^{-3}<\infty$.

For the reciprocal Gamma product let $D_-=[0]+\sum_{n\geq1}[-n]$ and
\[
 t_k=\left\lceil\frac{2\log k}{\log k-\log\log(k+1)}\right\rceil-1.
\]
Then
\[
 p_{D_-}=z(1+z)\prod_{k\geq2}E_{t_k}(-z/k),\qquad
 \beta_1(D_-)(z)=-z-\sum_{k\geq2}\sum_{j=2}^{t_k}\frac{(-z/k)^j}{j}.
\]
The term $-z$ is the small-shell conversion at $-1$.  The proved identity
is precisely
\[
 \Gamma(z)^{-1}=e^{\gamma z+\beta_1(D_-)(z)}p_{D_-}(z).
\]

\subsection{Recombining the Gamma multiplication factors}
For Gauss multiplication put
$P_{n,r}=\{-\ell-r/n:\ell\geq0\}$ and $P_n=\{-j/n:j\geq0\}$.
The common denominator in the four-step proof is
\begin{equation}\label{eq:gauss-annular}
 \begin{split}
 p_{P_n}(z)={}&z\prod_{j=1}^{2n-1}(1+nz/j)\\
 &\times\prod_{k\geq2}\prod_{j=nk}^{n(k+1)-1}
                  E_{t_k^{(n)}}(-nz/j),\\
 t_k^{(n)}={}&\left\lceil
 \frac{\log n+2\log k}{\log k-\log\log(k+1)}\right\rceil-1.
 \end{split}
\end{equation}
If instead the separate Gamma factors are first represented and then
multiplied ordinarily, their canonical denominators recombine as
\[
 p_{P_n}=e^{\delta_n}\prod_{r=0}^{n-1}p_{P_{n,r}},
\]
where
\[
 \delta_n(z)=\sum_{k\geq2}\sum_{j=nk}^{n(k+1)-1}
          \sum_{\ell=t_k^{(1)}+1}^{t_k^{(n)}}\frac{(-nz/j)^\ell}{\ell}.
\]
This is a finite polynomial: for fixed $n$, both genera equal $2$ for
all sufficiently large $k$, because
$\log n+3\log\log(k+1)<\log k$ eventually.  The four-step proof
represents the \emph{complete sides}, so this recombination is already
included in $h_s$ rather than being repeated in the comparison.

\section{Algebraic remarks on the representation}
\label{app:algebra}
This appendix retains the ring and field arguments.  They are not
additional prerequisites for DCT: the main text uses the ordinary field
$\MM$ and the original evaluation $\Phi(t)=q_t$ throughout.
The appendix is included as a complete algebraic companion to the proof
method, not as evidence for novelty of an abstract field. Its two
evaluations are distinguished explicitly below.

\subsection{The coordinate set and two multiplications}
Let $\DD$ be the group of locally finite signed divisors, and identify a
reduced triple with $([h],D)$, where $D=Z-P$.  Then
\[
 \TT=\{\zero\}\sqcup(\HH\times\DD),\qquad
 \Phi(\zero)=0,\quad \Phi([h],D)=e^h\kappa_D.
\]
Proposition~\ref{prop:coordinates} says that $\Phi$ is a bijection onto
$\MM$.  The direct coordinate multiplication is
\begin{equation}\label{eq:star}
 ([h],D)*([g],E)=([h+g],D+E),\qquad \zero*t=\zero.
\end{equation}
In triples one sums multiplicities and then takes positive and negative
parts, so common zeros and poles cancel.  The nonzero elements form an
abelian group, with identity $([0],0)$ and inverse $([-h],-D)$.
This follows immediately from the two additive coordinate groups.
The quotient $\HH$ is an additive group, not a quotient ring.

Define the normalized entire discrepancy $c(D,E)$ by
\begin{equation}\label{eq:cocycle}
 e^{c(D,E)}=\frac{\kappa_D\kappa_E}{\kappa_{D+E}},\qquad c(D,E)(0)=0.
\end{equation}
The quotient has the same leading coefficient $1$ at zero after the
common powers are removed, so the normalization is valid.  The function
$c$ is symmetric and satisfies
\[
 c(D,E)+c(D+E,F)=c(E,F)+c(D,E+F),\qquad
 c(D,0)=c(D,-D)=0.
\]
The exponentials on each side agree; evaluation at zero removes the
logarithmic constant.  This is the usual factor-set identity
\cite{Conrad}.

Ordinary multiplication of functions is instead
\begin{equation}\label{eq:ordinary-mult}
 ([h],D)\mathbin{\odot}([g],E)=([h+g+c(D,E)],D+E).
\end{equation}
Ordinary addition is obtained by forming
$F=e^h\kappa_D+e^g\kappa_E$.  If $F\equiv0$, the result is $\zero$;
otherwise it is
\[
 ([\log(F/\kappa_{D_3})],D_3),\qquad D_3=\Div(F).
\]
These operations form a field because they are transported through the
bijection $\Phi$ from the ordinary field $\MM$.
The divisor of a sum cannot in general be read from its two input
divisors: equal leading terms can cancel and new zeros can appear.

\begin{proposition}[Direct multiplication is not pointwise-distributive]
The law $*$ in \eqref{eq:star} is not distributive over ordinary
function addition in the $\Phi$-coordinates.
\end{proposition}
\begin{proof}
For shell $2$, the genus rule is $2$ for one point and $3$ for two
points counted with multiplicity.  For $L_a(z)=z-a$, $|a|=2$, the
logarithmic coordinate is
$[\log(-a)-z/a-z^2/(2a^2)]$.  Consequently
\[
 L_a*L_b=(z-a)(z-b)
       \exp\!\left(\frac{z^3}{3}(a^{-3}+b^{-3})\right),
\]
where the result is evaluated by $\Phi$.  Put $f=z-2$, so $f+4=z+2$.
Then
\[
 f*(f+4)=(z-2)(z+2),\qquad
 f*f+f*4=(z-2)^2e^{z^3/12}+4(z-2).
\]
The difference is $(z-2)^2(e^{z^3/12}-1)\not\equiv0$.
\end{proof}

\subsection{Algebraic corrections and their existence}
A compatible correction is an assignment $\beta:\DD\to\OO$ with
$\beta(D)(0)=0$ and
\begin{equation}\label{eq:beta-equation}
 \beta(D+E)=\beta(D)+\beta(E)+c(D,E).
\end{equation}
It makes $S(D)=e^{\beta(D)}\kappa_D$ multiplicative:
$S(D+E)=S(D)S(E)$.  Conversely, a normalized multiplicative section
supplies such a correction.  In particular $\beta(0)=0$ and
$\beta(-D)=-\beta(D)$.

\begin{proposition}[Algebraic existence]\label{prop:beta-exists}
A compatible correction exists for all locally finite signed divisors,
after an algebraic choice.
\end{proposition}
\begin{proof}
Choose a complex-linear projection $\Pi:\MM\to\OO$ fixing the entire
functions.  By extending vector-space bases it may also be chosen to
kill every rational function vanishing at infinity; this rational
subspace has zero intersection with $\OO$.
Set
\begin{equation}\label{eq:beta-projection}
 \beta_\Pi(D)(z)=-\int_0^z
       \Pi\!\left(\frac{\kappa_D'}{\kappa_D}\right)(w)\dd w.
\end{equation}
The integrand is entire, so the integral is path-independent.  On
differentiating \eqref{eq:cocycle},
\[
 c(D,E)'=\frac{\kappa_D'}{\kappa_D}
       +\frac{\kappa_E'}{\kappa_E}
       -\frac{\kappa_{D+E}'}{\kappa_{D+E}}.
\]
This combination is entire and is fixed by $\Pi$.  Thus the derivatives
of the two sides of \eqref{eq:beta-equation} agree, and both vanish at
zero before differentiation.  This proves the equation.
\end{proof}
The projection is an algebraic choice, not an algorithm continuous in
the divisor.  Each $\beta_\Pi(D)$ is nevertheless an entire function
of the plane variable.  Equivalently, the exact sequence
\[
 1\longrightarrow\OO^\times\longrightarrow\MM^\times
       \xrightarrow{\Div}\DD\longrightarrow0
\]
splits because $\OO^\times$ is divisible ($e^h$ has the root $e^{h/n}$)
and divisible abelian groups are injective \cite{Stacks}.

\subsection{A complete transported field}
Define
\begin{equation}\label{eq:Psi}
 \Psi_\beta(\zero)=0,\qquad
 \Psi_\beta([h],D)=e^{h+\beta(D)}\kappa_D=e^hS(D).
\end{equation}
This is a bijection to $\MM$, since it only shifts the logarithmic
coordinate in each divisor fiber.  Its multiplicative property is
\[
 \Psi_\beta(t*u)=\Psi_\beta(t)\Psi_\beta(u).
\]
Define addition by
\begin{equation}\label{eq:transport-add}
 t\oplus_\beta u=\Psi_\beta^{-1}
                  (\Psi_\beta(t)+\Psi_\beta(u)).
\end{equation}
Explicitly, put $F=\Psi_\beta(t)+\Psi_\beta(u)$.  If it is zero the
result is $\zero$.  Otherwise, with $D_3=\Div(F)$, the result is
\[
 \left([\log(F/\kappa_{D_3})-\beta(D_3)],D_3\right).
\]
The logarithm exists because its argument is an entire unit.

\begin{theorem}[Field and entire-function subring]\label{thm:field}
The structure $(\TT,\oplus_\beta,*)$ is a field isomorphic to $\MM$ by
$\Psi_\beta$.  Its subset
\[
 \mathscr A=\{\zero\}\cup\{([h],Z,\varnothing)\}
\]
is a subring corresponding to $\OO$, and its field of fractions is
$\TT$.
\end{theorem}
\begin{proof}
Associativity and commutativity of addition and its zero element are
transported from $\MM$ through the bijection.  The additive inverse of
$([h],D)$ is $([h+\pi\ii],D)$.  The multiplication identity is
$([0],0)$, and every nonzero element has inverse $([-h],-D)$.
For distributivity,
\begin{align*}
 \Psi_\beta(t*(u\oplus_\beta v))
 &=\Psi_\beta(t)\bigl(\Psi_\beta(u)+\Psi_\beta(v)\bigr)\\
 &=\Psi_\beta((t*u)\oplus_\beta(t*v)).
\end{align*}
Injectivity proves the law.  All other ring axioms follow in the same
way, or from the direct coordinate multiplication already established.

The map $\Psi_\beta$ preserves divisors, so a triple has no poles
exactly when its image is entire.  Therefore $\mathscr A$ is a subring.
Every nonzero $([h],Z,P)$ is the $*$-quotient of
$([h],Z,\varnothing)$ by $([0],P,\varnothing)$.  This proves the
fraction-field assertion.
\end{proof}
Different corrections give isomorphic presentations: the map
\[
 ([h],D)\longmapsto([h+\beta_1(D)-\beta_2(D)],D)
\]
intertwines the two evaluations.  This is not a new abstract field.
Importantly, the evaluation used here is $\Psi_\beta$, whereas the
DCT applications use $\Phi(t)=q_t$.  The transported addition is not
needed to prove any of those identities.

\subsection{Explicit corrections do not automatically give additive closure}
For fixed $m$, the summability group
\[
 \mathcal D_m=\left\{D:\sum_{a\ne0}|d(a)||a|^{-(m+1)}<\infty\right\}
\]
has the explicit correction \eqref{eq:beta-m}; its fixed-genus section
$s_m$ is multiplicative because its genus is independent of $D$.
Finite divisors admit in particular
\[
 \beta_0(D)=-\sum_{a\ne0}d(a)H_{g_D(a)}(z/a),
 \qquad s_0(D)=z^{d(0)}\prod_{a\ne0}(1-z/a)^{d(a)}.
\]
However, allowing arbitrary entire $h$ while restricting only
$D\in\mathcal D_m$ does not give a field under transported addition.
The zero-free functions $1$ and $e^{z^{m+1}}$ are both admitted, while
\[
 1+e^{z^{m+1}}=0\quad\Longleftrightarrow\quad
 z^{m+1}=(2k+1)\pi\ii,\qquad k\in\Z.
\]
Its simple zero multiset has divergent sum of $|a|^{-(m+1)}$ and is not
in $\mathcal D_m$.

A genuine restricted field can instead be obtained by pulling back an
already additively closed field, for example $\KK_\rho$ with $\rho<m+1$,
through $\Psi_m([h],D)=e^h s_m(D)$.  Its divisors lie in
$\mathcal D_m$ by zero-counting estimates, and the pullback restricts
$h$ and $D$ together.  Merely imposing a divisor summability condition
is not the same restriction.  These algebraic questions explain the
representation's scope; they do not alter the four-step DCT method.

\clearpage

\section*{Disclosure of AI assistance}
\addcontentsline{toc}{section}{Disclosure of AI assistance}
OpenAI ChatGPT was used as an assistance tool for drafting, algebraic checks, and locating related literature.  
The central mathematical ideas and the decision to organize the paper around fixed Weierstrass coordinates and divisor-comparison theorems are the author's. 
The author reviewed, revised, and takes responsibility for the mathematical statements, arguments, citations, and final contents of the manuscript.  ChatGPT is not listed as an author.

\phantomsection
\end{document}